\documentclass[11pt,reqno]{amsart}
\usepackage{amssymb}
\usepackage[all]{xy}
\usepackage{fancyhdr}
\usepackage{hyperref}
\usepackage{tikz-cd}
\usepackage{cite}
\usepackage{amsmath,amsfonts}
\usepackage{graphicx}
\usepackage{wrapfig}
\usepackage{array}
\usepackage{amsthm}
\usepackage[left=1.2in, right=1.2in, top=1in, bottom=1in]{geometry}
\usepackage{etoolbox}
\patchcmd{\section}{\scshape}{\bfseries}{}{}
\makeatletter
\renewcommand{\@secnumfont}{\bfseries}
\makeatother

\DeclareMathOperator{\Hom}{Hom}
\DeclareMathOperator{\Spec}{Spec}

\DeclareMathOperator{\Ker}{Ker}

\DeclareMathOperator{\Aut}{Aut}

\input xy
\xyoption{all}
\usepackage{secdot}  

\theoremstyle{plain}
\newtheorem{mydef}{\textbf{Definition}}[section]
\newtheorem{myeg}[mydef]{\textbf{Example}}
\newtheorem{mythm}[mydef]{\textbf{Theorem}}
\newtheorem*{nothm}{\textbf{Theorem}}

\newtheorem{rmk}[mydef]{\textbf{Remark}}
\newtheorem{lem}[mydef]{\textbf{Lemma}}
\newtheorem{pro}[mydef]{\textbf{Proposition}}

\newtheorem{que}[mydef]{\textbf{Question}}
\patchcmd{\abstract}{\scshape\abstractname}{\normalsize{\textbf{\abstractname}}}{}{}
\begin{document}

\title{Hyperstructures of affine algebraic group schemes}


\author{Jaiung Jun}
\address{Department of Mathematical Sciences, Binghamton University, Binghamton, NY 13902, USA}
\curraddr{}
\email{jjun@math.binghamton.edu}




\date{}

\dedicatory{}

\begin{abstract}
 \normalsize{\noindent We impose a rather unknown algebraic structure called a `hyperstructure' to the underlying space of an affine algebraic group scheme. This algebraic structure generalizes the classical group structure and is canonically defined by the structure of a Hopf algebra of global sections. This paper partially generalizes the result of A.Connes and C.Consani in \cite{con4}.}
\end{abstract}

\maketitle

\section{Introduction}
The idea of hyperstructures goes back to 1934 when F.Marty first suggested a notion of hypergroups in \cite{marty1935role} in such a way that a group multiplication is no longer single-valued but multi-valued. Shortly after, several aspects of hypergroups were investigated in relation to incidence geometry (see, \cite[\S 2.2]{corsini2003applications} for the historical development, also see \cite{con3} for the recent work of Connes and Consani in this direction).\\ 
In 1956, M.Krasner introduced a notion of hyperrings which generalizes commutative rings and use them in \cite{krasner1956approximation} for the approximation of valued fields. After Krasner's work, for decades, hyperstructures have been better known to computer scientists or applied mathematicians. This is due to uses of hyperstructures in connection with fuzzy logic (a form of multi-valued logic), automata, cryptography, coding theory via associations schemes, and hypergraphs (cf. \cite{corsini2003applications}, \cite{Dav2}, \cite{zieschang2006theory}). A notion of hypergroups has been also used in Harmonic analysis (cf. \cite{liltvinov}), however, algebraic aspects have not been much studied.\\ 
In recent years, the hyperstructure theory has been revitalized in connection with various fields. This is mainly done by Connes and Consani in connection to number theory, incidence geometry, and geometry in characteristic one (cf. \cite{con4}, \cite{con3}, \cite{con6}), O.Viro in connection to tropical geometry (cf. \cite{viro2}, \cite{viro}), and M.Marshall in connection to quadratic forms and real algebraic geometry (cf. \cite{mars2}, \cite{marshall2006real}). Furthermore, hyperstructures have certain relations with recently introduced algebraic objects such as supertropical algebras by Z.Izhakian and L.Rowen (cf. \cite{izhakian2014layered} , \cite{iza}), blueprints by O.Lorscheid (cf. \cite{oliver1}, \cite{lorscheid2015scheme}). Note that these are algebraic objects which aim to provide a firm algebraic foundation to tropical geometry. The author also applied an idea of hyperstructures to generalize the definition of valuations in \cite{jun2015valuations} and developed the basic notions of algebraic geometry over hyperrings in \cite{jaiungthesis}. 
\\

Let us now illustrate how a concept of hypergroups can be naturally implemented to affine algebraic group schemes. For an introduction to the basic notions of affine group schemes, we refer the readers to \cite{waterhouse2012introduction}.\\
Let $X=\Spec A$ be an affine algebraic group scheme over a field $k$. Then $A$ is a commutative Hopf algebra over $k$. Let $\Delta:A \longrightarrow A\otimes_k A$ be the coproduct and $m:A\otimes_k A \longrightarrow A$ be the multiplication. For a field extension $K$ of $k$, the set 
\[X(K)=\Hom(\Spec K, \Spec A)=\Hom(A,K)\] 
of $K$-rational points of $X$ has a group structure. More precisely, the group multiplication $*$ on the set $X(K)$ comes from the coproduct $\Delta$ of $A$ as follows: 
\begin{equation}\label{multiplicationofgroup}
f*g:=m \circ (f \otimes g)\circ \Delta, \quad f,g \in \Hom(A,K).
\end{equation}
However, in general, the underlying topological space $\Spec A$ itself does not carry any algebraic structure although $X$ is a group object in the category of affine schemes over $k$.\\ 
In the paper \cite{con4}, Connes and Consani adopted a notion of hyperstructures to recast the underlying topological space $\Spec A$ as a set of rational points of $X$ over the `Krasner's hyperfield' $\mathbf{K}$ (cf. Example \ref{krasner}). The novelty of their approach is that such a hyperstructure canonically arises from a coproduct of $A$. One of main ingredients of Connes and Consani is the following set bijection: 
\begin{equation}\label{bij}
\Hom(A,\mathbf{K})=\Spec A,
\end{equation}
where $\mathbf{K}$ is the Krasner's hyperfield and the homomorphisms are of hyperrings (by considering $A$ as a hyperring). In the view of \eqref{multiplicationofgroup} and \eqref{bij}, one is induced to ask if $\Spec A$ is a hypergroup. In \cite{con4}, Connes and Consani answered this question by generalizing the group multiplication of \eqref{multiplicationofgroup} to impose a hyperstructure to $\Spec A=\Hom(A,\mathbf{K})$. This algebraic (hyper) structure naturally emerges from a Hopf algebra structure of $A$. More precisely, Connes and Consani proved that if $A$ is a Hopf algebra over $\mathbb{Q}$ or $\mathbb{F}_p$ obtained from an affine line $\mathbb{G}_a$ or an algebraic torus $\mathbb{G}_m$, then $\Spec A$ is a (canonical) hypergroup (see, Theorem \ref{mainthemrem}).\\ 

In this paper, we first prove that Connes and Consani's definition is in fact an enrichment of the classical group structure as follows:
\begin{nothm}(cf. Proposition \ref{comparison})
Let $A$ be a Hopf algebra over a field $k$ with $|k|\geq 3$, $K$ be a field extension of $k$, and $X=\Spec A$. Then we have the following injection (of sets) 
\[
i: X(K)=\Hom(A,K) \longrightarrow X=\Spec A\]
such that $i(f*g)\subseteq i(f)*_hi(g)$,
where $*$ is the group multiplication of $X(K)$ and $*_h$ is the hyperoperation of $X$. 
\end{nothm}
Then, we partially generalize their result to arbitrary affine algebraic group schemes as follows.
\begin{nothm}(cf. Theorem \ref{chevalleytheorem})
Any affine algebraic group scheme $X=\Spec A$ over a field $k$, such that $|k|\geq 3$, has a canonical hyperstructure $*$ induced from the coproduct on $A$ which satisfies the following conditions:
\begin{enumerate}
\item
$*$ is weakly-associative, i.e. $f*(g*h)\cap (f*g)*h \neq \emptyset$ for $\forall f,g,h \in X$. 
\item
$*$ is equipped with the identity element $e$, i.e. $f*e=e*f=f$ for $\forall f\in X$.
\item
For each $f \in X$, there exists a canonical element $\tilde{f} \in X$ such that $e \in (f*\tilde{f})\cap (\tilde{f} * f)$. 
\item
For $f,g,h \in X$, the following holds: $f \in g*h \Longleftrightarrow \tilde{f} \in \tilde{h}*\tilde{g}$.
\end{enumerate}
\end{nothm}

\subsection*{Acknowledgment}
This paper is a part of the author's Ph.D. thesis \cite{jaiungthesis}. The author expresses his deep gratitude to his academic advisor Caterina Consani for countless conversations on this project. The author also thanks to Oliver Lorscheid for reading the first draft of this paper and provided helpful comments. In particular, Proposition \ref{comparison} comes from his observation. 

\section{Basic notions of hypergroups and hyperrings}
In this section, we provide the basic definitions of hypergroup and hyperring theory. For a complete introduction, we refer the readers to \cite{corsini2003applications}. 
\subsection{Hypergroups}
\begin{mydef}
Let $H$ be a nonempty set and $\mathcal{P}(H)$ be the set of nonempty subsets of $H$.
\begin{enumerate} 
\item
A hyperoperation on $H$ is a function, $* : H \times H \rightarrow \mathcal{P}(H)$. 
\item
For any nonempty subsets $A, B\subseteq H$, we define 
\[
A*B:=\bigcup_{a\in A, b\in B}(a*b).
\]  
\end{enumerate}
When $a*b$ contains a single element $c$, we write $a*b=c$ instead of $a*b=\{c\}$ for simplicity. 
\end{mydef}

\begin{mydef}\label{hypergroup}
A hypergroup $(H,*)$ is a nonempty set $H$ with a hyperoperation $*$ which satisfies the following properties:
\begin{enumerate}
\item
$(a*b)*c=a*(b*c)\quad  \forall a,b,c\in H.$
\item
$\exists !$ $e \in H$ such that $e*a=a*e=a$ for all $a \in H$.
\item
For each $a \in H$, $\exists !$ $b$ $(:=a^{-1})$ such that $e \in(a*b) \cap (b*a)$.
\end{enumerate}
We call the element $e$ of $H$ the identity element.
\end{mydef}

\begin{rmk} 
In fact, our Definition \ref{hypergroup} is stronger than the first definition given by Marty. In \cite{marty1935role}, a hypergroup is a nonempty set $H$ together with a hyperoperation $*$ which satisfies: $(a*b)*c=a*(b*c)\quad  \forall a,b,c\in H$ and $a*H=H*a=H$. One can easily observe that if $(H,*)$ is a hypergroup in the sense of Definition \ref{hypergroup}, then $(H,*)$ is a hypergroup in the sense of Marty. 
\end{rmk}

When a hypergroup $(H,*)$ is commutative (i.e., $a*b=b*a$), we call $(H,*)$ a canonical hypergroup. In this case, $(H,*)$ satisfies the following property (reversibility):
\[c \in a*b \Longrightarrow b \in c*a^{-1}\textrm{ and }a \in c*b^{-1}.\]
In case of a canonical hypergroup, we use $+$ notation for a hyperoperation.
\begin{myeg}\label{krasnergroup}
Let $\mathbf{K}:=\{0,1\}$. Then $(\mathbf{K},+)$ becomes a canonical hypergroup under the following commutative hyperoperation:
\[0+1=1=1+0,\quad  0+0=0, \quad  1+1=\{0,1\}.\]
\end{myeg}
\begin{myeg}\label{signgroup}
Let $\mathbf{S} = \{-1,0,1\}$. One may impose a commutative hyperoperation $+$ following the rule of signs:
\[0+0=0,\quad 1+0=1=1+1,\quad  (-1)+0=(-1)=(-1)+(-1), \quad 1+(-1)=\{ -1,0,1\}.\]
Then $(\mathbf{S},+)$ is a canonical hypergroup.
\end{myeg}

\begin{mydef}
Let $(H_1,*_1)$ and $(H_2,*_2)$ be hypergroups. A homomorphism $f$ from $H_1$ to $H_2$ is a function $f:H_1\longrightarrow H_2$ such that $f(e_1)=f(e_2)$ and
\[f(a*_1b)\subseteq f(a)*_2f(b), \quad \forall a,b \in H_1,\]
where $e_1$ and $e_2$ are identity elements of $H_1$ and $H_2$.
When $f(a*_1b)=f(a)*_2f(b)$ for all $a,b \in H_1$, $f$ is said to be a strict homomorphism. 
\end{mydef}

\subsection{Hyperrings}
In this subsection, we review the basic definitions of hyperring theory. We will restrict ourselves to Krasner hyperring. What it follows, by a hyperring we will always mean a Krasner hyperring.

\begin{mydef}
A (Krasner) hyperring $(R, + ,\cdot)$ is a nonempty set $R$ with a hyperoperation $+$ and a binary operation $\cdot$ which satisfy the following conditions:
 \begin{enumerate}
\item
$(R,+,0)$ is a canonical hypergroup, where $0$ is the identity element.
\item
$(R,\cdot,1)$ is a commutative monoid, where $1$ is the identity element.
\item
$ \forall a, b, c \in R$,  $ a\cdot(b+c)=a\cdot b+a\cdot c$,  $(a+b)\cdot c=a\cdot c+b\cdot c$.
\item
$\forall a \in R$,  $ a\cdot 0 = 0 = 0\cdot a$.
\item
$0 \neq 1$.
\end{enumerate}
When $(R\setminus \{ 0 \}, \cdot)$ is a group, we call $(R, +, \cdot)$ a hyperfield.
\end{mydef}
\begin{mydef}
Let $(R_1, +_1, \cdot_1)$,  $(R_2, +_2, \cdot_2)$ be hyperrings. A function $ f : R_1 \longrightarrow R_2$ is said to be a homomorphism of hyperrings if
\begin{enumerate}
\item
$f$ is a homomorphism of canonical hypergroups $(R_1,+_1)$ and $(R_2,+_2)$.
\item
$f$ is a homomorphism of monoids $(R_1,\cdot_1)$ and $(R_2,\cdot_2)$.
\item
$f$ is said to be strict if $f$ is strict as a homomorphism of canonical hypergroups.
\end{enumerate}
\end{mydef}

\begin{myeg}\label{krasner}
Let $\mathbf{K}:=\{0,1\}$. We impose a commutative monoid structure on $\mathbf{K}$ as follows:
\[1\cdot 1=1, \quad 0\cdot 1=0=1 \cdot 0,\] 
One can observe that this monoid structure is compatible with the canonical hypergroup structure given in Example \ref{krasnergroup}. In fact, $(\mathbf{K},+,\cdot)$ becomes a hyperfield called the Krasner's hyperfield.
\end{myeg}
\begin{myeg} 
Let $\mathbf{S} = \{-1,0,1\}$. One may impose a commutative monoid structure on $\mathbf{S}$ as follows:
\[1\cdot 1=1=(-1)\cdot(-1),\quad  (-1)\cdot 1=(-1),\quad  1\cdot 0=0=0\cdot (-1)=0\cdot 0.\]
Then, together with a canonical hypergroup structure given in Example \ref{signgroup}, $\mathbf{S}$ becomes a hyperfield called the hyperfield of signs.
\end{myeg}

We close this subsection by providing the following theorem of Connes and Consani which asserts that we have a rich class of hyperrings. 

\begin{mythm}(cf. \cite[Proposition $2.6$]{con3}) \label{orbit}
Let $A$ be a commutative ring and $A^\times$ be the group of (multiplicatively) invertible elements of $A$. Then, for any subgroup $G$ of $A^{\times}$, the set $A/G=\{aG \mid a \in A\}$ of cosets has a hyperring structure with the following operations:
\begin{enumerate}
\item
(multiplication): $aG\cdot bG := abG, \quad \forall aG, bG \in A/G$. 
\item
(hyperaddition): $aG+bG:= \{cG \mid c=ax+by \textrm{ for some }x,y \in G\},\quad \forall aG, bG \in A/G$.  
\end{enumerate}
A hyperring of this type is called a quotient hyperring.
\end{mythm}
In this way, we can see that the Krasner's hyperfield $\mathbf{K}$ is isomorphic to the quotient hyperring $k/k^\times$ for any field $k$ with $|k|\geq 3$.

\section{Hyperstructure of affine algebraic group schemes}
We first review how Connes and Consani generalize the group operation \eqref{multiplicationofgroup} to hyperstructures in \cite{con4}.

\begin{mydef}(\cite[Definition $6.1$]{con4}\label{hyperop})
Let $(A,\Delta)$ be a commutative ring with a coproduct $\Delta:A \longrightarrow A\otimes_{\mathbb{Z}}A$ and let $R$ be a hyperring. Let $X=\Hom(A,R)$ be the set of homomorphisms of hyperrings (by considering $A$ as a hyperring). For $\varphi_j \in X$, $j=1,2$, one defines
\begin{equation}\label{hyperoperation}
\varphi_1*_\Delta \varphi_2:=\{\varphi \in X \mid \varphi(x)\in\sum \varphi_1(x_{(1)})\varphi_2(x_{(2)}) ,\quad \forall \Delta(x)=\sum x_{(1)}\otimes x_{(2)}\}.
\end{equation}
Note that, in general, $\Delta(x)$ can have many presentations as an element of $A\otimes_{\mathbb{Z}} A$, and the condition in \eqref{hyperoperation} should hold for all presentations of $\Delta(x)$.
\end{mydef}

\begin{rmk}
One can easily notice that when $(A,\Delta)$ is cocommutative, the hyperoperation as in \eqref{hyperoperation} is commutative.
\end{rmk}

The following lemma of Connes and Consani will be used in sequel.
\begin{lem}(\cite[Lemma $6.4$]{con4}\label{lemcc})
Let $(A,\Delta)$ be a commutative ring with a coproduct $\Delta:A \longrightarrow A\otimes_{\mathbb{Z}}A$ and $J_j$ be ideals of $A$ for $j=1,2$. Then, the set 
\begin{equation}\label{J=J_1HHJ_2}
J:=J_1\otimes_{\mathbb{Z}} A + A\otimes_{\mathbb{Z}} J_2 
\end{equation}
is an ideal of $A\otimes_{\mathbb{Z}} A$ as well as the set 
\begin{equation}\label{inverofdelta}
J_1 *_\Delta J_2:=\{x \in A \mid\Delta(x) \in J\}
\end{equation}
is an ideal of $A$. Furthermore, for $\varphi \in \varphi_1*_{\Delta} \varphi_2$, we have
\begin{equation}\label{kerneldelta}
\Ker(\varphi_1)*_\Delta \Ker(\varphi_2) \subseteq \Ker(\varphi).
\end{equation}
\end{lem}

\noindent In \cite{con4}, the authors proved that for a commutative ring $A$ and for the Krasner's hyperfield $\mathbf{K}$, one has the following identification (of sets):
\begin{equation}\label{setsame}
\Hom(A,\mathbf{K})=\Spec A, \quad \varphi \mapsto \Ker(\varphi).
\end{equation}
Thus, the underlying topological space $\Spec A$ can be considered as the set of `$\mathbf{K}$-rational points' of the affine scheme $X=\Spec A$. The following theorem is the main motivation of the paper.

\begin{mythm}(\cite[Theorems $7.1$ and $7.13$]{con4}\label{mainthemrem})
Let $\mathbf{K}$ be the Krasner's hyperfield.
\begin{enumerate}
\item
Let $\delta$ be the generic point of $\Spec\mathbb{Q}[T]$. Then, $\Spec \mathbb{Q}[T] \backslash \{\delta\}$ and $ \Spec \mathbb{Q}[T,\frac{1}{T}] \backslash \{\delta\}$ are hypergroups via \eqref{hyperoperation} and \eqref{setsame}. Moreover, we have 
\[\Spec \mathbb{Q}[T] \backslash \{\delta\} \simeq \bar{\mathbb{Q}}/\Aut(\bar{\mathbb{Q}}), \quad \Spec \mathbb{Q}[T,\frac{1}{T}] \backslash \{\delta\} \simeq \bar{\mathbb{Q}}^{\times}/\Aut(\bar{\mathbb{Q}}).\]
\item
Let $\Omega$ be an algebraic closure of the field of fractions, $\mathbb{F}_p(T)$. Then, $\Spec \mathbb{F}_p[T]$ and $ \Spec \mathbb{F}_p[T,\frac{1}{T}]$ are hypergroups via \eqref{hyperoperation} and \eqref{setsame}. We also have 
\[\Spec \mathbb{F}_p[T]\simeq \Omega /\Aut(\Omega), \quad \Spec \mathbb{F}_p[T,\frac{1}{T}]\simeq \Omega^{\times} /\Aut(\Omega).\]
\end{enumerate}
\end{mythm}
\begin{rmk}
Note that the hypergroup structure of $\bar{\mathbb{Q}}/\Aut(\bar{\mathbb{Q}})$, $\bar{\mathbb{Q}}^{\times}/\Aut(\bar{\mathbb{Q}})$, $\Omega /\Aut(\Omega)$, and $\Omega^{\times} /\Aut(\Omega)$ are given similar to Theorem \ref{orbit}. For details, see \cite{con4}.
\end{rmk}
In other words, the Connes and Consani defined the hyperoperation $*$ on $X=\Spec A$ when $A$ is a commutative ring with a coproduct and showed that in some cases, $(X,*)$ is a hypergroup (cf. Theorem \ref{mainthemrem}). 
In this paper, we show that $(X=\Spec A,*)$ is an algebraic object which is more general than a hypergroup. In what follows, by $\mathbf{K}$ we always mean the Krasner's hyperfield (cf. Example \ref{krasner}). Also note that in general, we can not expect the hyperoperation $*$ on $X=\Spec A$ to be commutative unless $A$ is cocommutative.
 \begin{rmk}\label{fiber}
Suppose that $A$ is a commutative ring with a coproduct $\Delta$. For $f,g \in \Hom(A,\mathbf{K})$, unless $f|_\mathbb{Z}=g|_\mathbb{Z}$, $f*g$ is an empty set (\cite[Lemma $6.2$]{con4}). In other words, the hyperoperation $*$ is non-trivial only within the fibers of the following restriction map
\[ \Phi: \Hom(A,\mathbf{K}) \rightarrow \Hom(\mathbb{Z},\mathbf{K})=\Spec \mathbb{Z},\quad  f\mapsto f|_\mathbb{Z}. \]
As explained in \cite{con4}, one can easily check that for the generic point $\delta \in \Spec \mathbb{Z}$, we have the identification $\Phi^{-1}(\delta)=\Hom(A\otimes_{\mathbb{Z}}\mathbb{Q},\mathbf{K})$ which is compatible with the hyperoperations. Also, for $\wp=(p) \in \Spec \mathbb{Z}$, we have the identification $\Phi^{-1}(\wp)=\Hom(A\otimes_{\mathbb{Z}}\mathbb{F}_p,\mathbf{K})$ which is also compatible with the hyperoperations.
\end{rmk}
In the view of Remark \ref{fiber}, in the following, we will focus on the case of a commutative Hopf algebra over a field $k$. Also, in the sequel, all Hopf algebras will be assumed to be commutative.\\
We begin with a lemma showing that if we work over a field, our hyperoperation is always non-trivial. 

\begin{lem}\label{nontrivialop}
Let $A$ be a Hopf algebra over a field $k$ with a coproduct $\Delta: A \rightarrow A \otimes_{k} A$. If $f,g \in \Hom(A,\mathbf{K})$, then the set
\[ P:=\Delta^{-1}(\Ker(f) \otimes_k A + A \otimes_k \Ker(g)) \]
is a prime ideal of $A$.
\end{lem}

\begin{proof}
Trivially, $P$ is an ideal by being an inverse image of an ideal. Hence, all we have to show is that $P$ is prime. Suppose that $\alpha \beta \in P$. Then, by definition, $\Delta(\alpha \beta) \in \Ker(f) \otimes_k A + A \otimes_k \Ker(g)$. This implies that for any decomposition $\Delta(\alpha\beta)=\sum \gamma_{(1)} \otimes_k \gamma_{(2)}$, we have $\sum f(\gamma_{(1)})g(\gamma_{(2)})=0$. Assume that $\alpha \not \in P$. Then, there is a decomposition $\Delta\alpha=\sum a_i \otimes_k b_i$ such that $\sum f(a_i)g(b_i) =1$ or $\{0,1\}$. If $\beta \not \in P$, then we also have a decomposition $\Delta\beta=\sum c_j \otimes_k d_j$ such that $\sum f(c_j)g(d_j) =1$ or $\{0,1\}$. For these two specific decompositions, we have
\begin{equation}\label{ab_comute}
 \Delta(\alpha\beta)=\Delta(\alpha)\Delta(\beta)=(\sum a_i \otimes_k b_i)(\sum c_j \otimes_k d_j) = \sum_{i,j} a_ic_j \otimes_k b_id_j.
\end{equation}
Since $\alpha\beta \in P$, we should have
\begin{multline}\label{ab}
\sum_{i,j}f(a_ic_j)g(b_id_j)=\sum_{i,j}f(a_i)f(c_j)g(b_i)g(d_j)\\
=\sum_{i,j}f(a_i)g(b_i)f(c_j)g(d_j)=\sum_i[(f(a_i)g(b_i)) \sum_j f(c_j)g(d_j)]=0.
\end{multline}
However, since we know that $\sum_i f(a_i)g(b_i) =1 \textrm{ or } \{0,1\}$ and $\sum_j f(c_j)g(d_j) =1 \textrm{ or } \{0,1\}$, we only can have
\[ \sum_i[(f(a_i)g(b_i)) \sum_j f(c_j)g(d_j)]=1 \textrm{ or } \{0,1\}. \]
This contradicts to \eqref{ab}. Hence, either $\alpha$ or $\beta$ should be in $P$.
\end{proof}

\begin{lem}\label{nonempty}
Let $A$ be a Hopf algebra over a field $k$. If $f,g \in \Hom(A,\mathbf{K})$, then the set $f*g$ is not empty.
\end{lem}

\begin{proof}
We use the same notation as in Lemma \ref{nontrivialop}. For a non-zero element $a \in k$, we have $f(a)=g(a)=1$. It follows that $k \not \subseteq P$ and hence $P \neq A$. Thus, in this case, $P$ is a proper prime ideal. From the identification $\Hom(A,\mathbf{K})=\Spec A$ of \eqref{setsame}, we have the homomorphism $\varphi:A \rightarrow \mathbf{K}$ of hyperrings such that $\Ker(\varphi)=P$. We claim that $\varphi \in f*g$. Indeed, let $\alpha \in A$. First, suppose that $\alpha \in P$. Then, $\varphi(\alpha)=0$. On the other hand, for any decomposition $\Delta(\alpha)=\sum a_i \otimes b_i$, we have $\sum f(a_i)g(b_i)=0$ since $\alpha \in P$. When $\alpha \not \in P$, we have $\varphi(\alpha)=1$. However, In this case, $\sum f(a_i)g(b_i)=1$ or $\{0,1\}$ in this case. This proves that $\varphi \in f*g$.
\end{proof}

\begin{rmk}
Under the same notation as Lemma \ref{nontrivialop}, we consider the case of a commutative $A$ with a coproduct $\Delta$. Let $p$ and $q$ be distinct prime numbers. Suppose that $p \in \Ker(f)$ and $q \in \Ker(g)$ for some $f,g \in \Hom(A,\mathbf{K})$. Then, one can easily see that $p,q \in P$. This implies that $1 \in P$ and hence $P=A$. Furthermore, for $\varphi \in f*g$, we have $P \subseteq \Ker(\varphi)$ from Lemma \ref{lemcc}. It follows that the only possible element $\varphi$ in $f*g$ is the zero map since $P=A$. However, this is impossible since $\varphi(1)=1$. Thus, in this case, we have $f*g=\emptyset$ as previously mentioned in Remark \ref{fiber}.
\end{rmk}

Next, we prove that the hyperstructure which Connes and Consani defined is an enrichment of the classical group structure. 

\begin{pro}\label{comparison}
Let $A$ be a Hopf algebra over a field $k$ with $|k|\geq 3$, $K$ be a field extension of $k$, and $X=\Spec A$. There is an injection $i$ from $X(K)=\Hom(A,K)$ to $X=\Spec A$ such that
\[i(f*g)\subseteq i(f)*_hi(g),\]
where $*$ is the group multiplication of $X(K)$ and $*_h$ is the hyperoperation of $X$. 
\end{pro}
\begin{proof}
Define $i$ as follows:
\begin{equation}
i: X(K) \longrightarrow X, \quad \varphi \mapsto \Ker(\varphi).
\end{equation}
Then this map is clearly injective. Suppose that $h=f*g$. We use the set bijection \eqref{setsame} and consider $X$ as $\Hom(A,\mathbf{K})$, where $\mathbf{K}$ is the Krasner's hyperfield. Then the above map becomes:
\begin{equation}
i: \Hom(A,K)\longrightarrow \Hom(A,\mathbf{K}),\quad \varphi \mapsto \pi\circ\varphi,
\end{equation}
where $\pi:K\longrightarrow K/K^\times =\mathbf{K}$ is the canonical projection. For the notational simplicity, let $i(f)=\tilde{f}$ for each $f \in \Hom(A,K)$. Now, for any $a \in A$ and $\Delta(a)=\sum a_i\otimes b_i$, we want to show that
\begin{equation}\label{ineq}
\tilde{h}(a)\in \sum \tilde{f}(a_i)\tilde{g}(b_i).
\end{equation}
Let us first consider the case when $\tilde{h}(a)=0$. This means that $a \in \Ker(h)$. It follows that
\[
\sum f(a_i)g(b_i)=0.
\]
Then either $f(a_i)g(b_i)=0$ for all indexes $i$ or there are at least two indexes $j,l$ such that $f(a_j)g(b_j)\neq 0$ and $f(a_l)g(b_l)\neq 0$. In the first case, we obtain $\sum \tilde{f}(a_i)\tilde{g}(b_i)=0$ and the second case, we obtain $\sum \tilde{f}(a_i)\tilde{g}(b_i)=\{0,1\}$. Thus, in any case, we have \eqref{ineq}.\\
Next, suppose that $\tilde{h}(a)=1$. This implies that $h(a) \neq 0$. Since $h(a)=\sum f(a_i)g(b_i)$, it follows that either $f(a_r)g(b_r) \neq 0$ for exactly one index $r$ or there are at least two indexes $j,l$ such that $f(a_j)g(b_j)\neq 0$ and $f(a_l)g(b_l)\neq 0$. But, in any case, we have \eqref{ineq}. This completes our proof. 
\end{proof}

\begin{rmk}
Proposition \ref{comparison} also implies Lemma \ref{nonempty}.
\end{rmk}
The following proposition shows that the hyperoperation of an affine algebraic group scheme $X$ descends to a closed subgroup scheme. In the sequel, we always assume that any field $k$ contains more than two elements. 
\begin{pro}\label{inducedfromgeneral}
Let $A$ be a finitely generated Hopf algebra over a field $k$. Let $H$ be a closed subgroup scheme of the affine algebraic group scheme $G=\Spec A$ and let $B:=\Gamma(H,\mathcal{O}_H)$ be the Hopf algebra of global sections of $H$. Then, there exists an injection (of sets):
\[ \sim: \Hom(B,\mathbf{K}) \hookrightarrow \Hom(A,\mathbf{K}) \]
which preserves the hyperoperations. i.e., for $f,g \in \Hom(B,\mathbf{K})$, we have
\begin{equation}\label{compatible}
\widetilde{f \star g}=\tilde{f}*\tilde{g},
\end{equation}
where $\star$ is the hyperoperation on $\Hom(B,\mathbf{K})$ and $*$ is the hyperoperation on $\Hom(A,\mathbf{K})$ as in Definition \ref{hyperop}. 
\end{pro}

\begin{proof}
Since $H$ is a closed subgroup scheme of $G$, we know that $B \simeq A/I$ for some Hopf ideal $I$ of $A$. Consider the following set:
\[X_I=\{f \in \Hom(A,\mathbf{K}) \mid f(i)=0\quad  \forall i \in I\}.\]
Let $\pi : A \rightarrow A/I$ be a canonical projection map. We define the following map:
\[\sim : \Hom(B,\mathbf{K})=\Hom(A/I,\mathbf{K}) \longrightarrow X_I,\quad \varphi \mapsto \tilde{\varphi}, \]
where $\tilde{\varphi}$ is an element of $\Hom(A,\mathbf{K})$ such that $\Ker(\tilde{\varphi}):=\pi^{-1}(\Ker\varphi)$. Note that from the identification \eqref{setsame}, the map $\sim$ is well defined. Furthermore, since there is an one-to-one correspondence between the set of prime ideals of $A$ containing $I$ and the set of prime ideals of $B\simeq A/I$ given by $\wp \mapsto \wp/I$, the map $\sim$ is a bijection (of sets). We remark the following two facts:
\begin{enumerate}
\item  \label{fact}
If $\varphi \in \Hom(A/I,\mathbf{K})$ then $\tilde{\varphi}(r)=\varphi([r])$ for $r \in A$, where $[r]=\pi(r)$. In other words, $\tilde{\varphi}=\varphi\circ \pi$. In fact, since $\Ker\varphi=\Ker(\tilde{\varphi}) / I$, we have 
\[\tilde{\varphi}(r)=0 \Longleftrightarrow r \in \Ker(\tilde{\varphi}) \Longleftrightarrow \varphi([r])=\varphi(r/I)=0.\]
\item\label{hopfideal}
For $\tilde{f}$, $\tilde{g} \in X_I$, we have $\tilde{f}*\tilde{g} \subseteq X_I$. Indeed, suppose that $\phi \in \tilde{f}*\tilde{g}$. Then, we have to show that
for $i \in I$, $\phi(i)=0$. However, since $I$ is a Hopf ideal, we have
\[\Delta(I) \subseteq I \otimes_k A + A \otimes_k I.\]
This implies that $\phi(i) \in \sum \tilde{f}(i_{(1)})\tilde{g}(i_{(2)}) =\{0\}$ for any decomposition $\Delta(i)=\sum i_{(1)} \otimes_k i_{(2)}$ since $\tilde{f}(a)=\tilde{g}(a)=0$ $\forall a \in I$.
\end{enumerate}
Next, we prove that the map $\sim$ is compatible with the hyperoperations; $\widetilde{f\star g}=\tilde{f}*\tilde{g}.$\\
Let $\Delta_A$ be a coproduct of $A$ and $\Delta_I$ be a coproduct of $B\simeq A/I$. Suppose that $\varphi \in f\star g$ and let $\Delta_A(r)=\sum r_{(1)} \otimes r_{(2)}$ be a decomposition of $r \in A$. We have to show that
\[ \tilde{\varphi}(r) \in \sum \tilde{f}(r_{(1)}) \tilde{g}(r_{(2)}).\]
Since $I$ is a Hopf ideal, we have the following commutative diagram:
\begin{equation}\label{hopfdiagram}
\begin{tikzcd}
 A  \ar{r}{\Delta_A} \ar{d}{\pi} & A \otimes_k A \ar{d}{\pi \otimes \pi}    \\
 A/I  \ar{r}{\Delta_{I}} & A/I \otimes_k A/I
\end{tikzcd}
\end{equation}
It follows that $\Delta_I([r])= \sum [r_{(1)}] \otimes_k [r_{(2)}]$. However, since $\varphi \in f\star g$, we have
\[\varphi([r]) \in \sum f([r_{(1)}]) g([r_{(2)}]).\]
From the above remark (\ref{fact}), this implies that $\tilde{\varphi}(r) \in \sum \tilde{f}(r_{(1)}) \tilde{g}(r_{(2)})$. Hence, $\tilde{\varphi} \in \tilde{f}*\tilde{g}$.\\
Conversely, let $\tilde{f}, \tilde{g} \in X_I$ and suppose that $\psi \in \tilde{f}*\tilde{g}$. Since $\sim$ is a bijection, from the above remark \ref{hopfideal}, $\psi=\tilde{\varphi}$ for some $\varphi \in \Hom(B,\mathbf{K})$. We claim that $\varphi \in f \star g$. In other words, for $[r] \in A/I$ and a decomposition $ \Delta_I ([r]) = \sum [r_{(1)}] \otimes_k [r_{(2)}]$, we show that
\[\varphi([r]) \in \sum f([r_{(1)}]) g([r_{(2)}]). \]
Since $\pi$ is surjective, we have $\Ker(\pi \otimes_k \pi)\subseteq \Ker \pi \otimes_k A + A\otimes_k \Ker\pi$. Therefore, from \eqref{hopfdiagram}, we can find the following decomposition of $r$:
\[\Delta_A (r) = \sum r_{(1)} \otimes_k r_{(2)} + \sum i_{(1)} \otimes_k a_{(2)} + \sum a_{(1)} \otimes_k i_{(2)},\]
where $i_{(1)}, i_{(2)} \in I$ and $a_{(1)},a_{(2)} \in A$. Since $\tilde{\varphi} \in \tilde{f}*\tilde{g}$, we have
\[ \tilde{\varphi}(r) \in \sum \tilde{f}(r_{(1)}) \tilde{g}(r_{(2)}) + \sum \tilde{f}(i_{(1)}) \tilde{g}(a_{(2)}) + \sum \tilde{f}(a_{(1)}) \tilde{g}(i_{(2)}).\]
However, it follows from the definition of $\tilde{f}$, $ \tilde{g} \in X_I$ that 
\[\sum \tilde{f}(i_{(1)}) \tilde{g}(a_{(2)})= \sum \tilde{f}(a_{(1)}) \tilde{g}(i_{(2)})=0.\]
Therefore, we have $\tilde{\varphi}(r) \in \sum \tilde{f}(r_{(1)}) \tilde{g}(r_{(2)})$. From the above remark (\ref{fact}), this implies that $\varphi([r]) \in \sum f([r_{(1)}]) g([r_{(2)}])$. Hence, $\varphi \in f \star g$.
\end{proof}

\begin{myeg}\label{mu2}
Let $A:=\mathbb{Q}[T]/(T^2-1)$. It follows from Theorem \ref{mainthemrem} and Proposition \ref{inducedfromgeneral} that the hyperstructure of $\Spec A$ should be induced from the hyperstructure of $\bar{\mathbb{Q}}^\times/\Aut (\bar{\mathbb{Q}})$. Therefore, in this case, the hyperstructure of $\Spec A$ coincides with the group structure of $\mu_2(\mathbb{Q})$.
\end{myeg}

\begin{myeg}
Let $A:=\mathbb{F}_p[T]/(T^{p-1}-1)$. Then similar to Example \ref{mu2}, one can see that the hyperstructure of $\Spec A$ is in fact the group structure of $\mu_{p-1}(\mathbb{F}_p)$.
\end{myeg}

\noindent Let $GL_n$ be the general linear group scheme over a field $k$ such that $|k| \geq 3$. We will prove the following statements:
\begin{enumerate}
\item \label{statement1}
The hyperstructure $*$ on $GL_n(\mathbf{K})$ as in Definition \ref{hyperop} is weakly-associative.
\item \label{statement2}
The identity of $(GL_n(\mathbf{K}),*)$ is given by $e= \varphi  \circ \varepsilon $, where $\varepsilon$ is the counit of the Hopf algebra $\mathcal{O}_{GL_n}$ and $\varphi: k \rightarrow k/k^{\times}=\mathbf{K}$ is a canonical projection map. 
\item \label{statement3}
For $f \in GL_n(\mathbf{K})$, a canonical inverse $\tilde{f}$ of $f$ is given by $\tilde{f}= f \circ S$, where $S:\mathcal{O}_{GL_n} \longrightarrow \mathcal{O}_{GL_n}$ is the antipode map. Furthermore, we have
\[f \in h*g \Longleftrightarrow \tilde{f} \in \tilde{g}*\tilde{h}.\] 
\end{enumerate}

\noindent Any affine algebraic group scheme $G$ is a closed subgroup scheme of a group scheme $GL_n$ for some $n \in \mathbb{N}$. Assume that the above statements are true. Then, from Proposition \ref{inducedfromgeneral}, we can derive that the set $G(\mathbf{K})$ of `$\mathbf{K}$-rational points' of an affine algebraic group scheme $G$ has the hyperstructure induced from $GL_n$ which is weakly-associative equipped with a canonical inverse (not unique) and the identity, and also satisfies the inversion property.\\ 

\noindent In what follows, we let $A=\mathcal{O}_{GL_n}=k[X_{11},X_{12},...,X_{nn},1/d]$ be the Hopf algebra of the global sections of the general linear group scheme $GL_n$ over a field $k$ such that $|k|\geq 3$, where $d$ is the determinant of an $n \times n$ matrix. We first prove the statement (\ref{statement2}). Note that we impose the condition $|k|\geq 3$ so that we can realize the Krasner's hyperfield $\mathbf{K}$ as $k/k^\times$ (cf. Theorem \ref{orbit}).
\begin{lem}\label{neutralelt}
The identity of the hyperoperation $*$ on $\Hom(A,\mathbf{K})$ is given by $e= \varphi  \circ \varepsilon $, where $\varepsilon$ is the counit of $A=\mathcal{O}_{GL_n}$ and $\varphi: k \rightarrow k/k^{\times}=\mathbf{K}$ is a canonical projection map.
\end{lem}

\begin{proof}
Let $f \in \Hom(A,\mathbf{K})$. We first claim that $f \in e*f$. Indeed, let $P \in A$. Then, for a decomposition $\Delta P = \sum a_i \otimes_k b_i$, we have $P=\sum \varepsilon(a_i)b_i$ since $\varepsilon$ is the counit. It follows that 
\[f(P)=f(\sum \varepsilon(a_i)b_i) \in \sum f(\varepsilon(a_i)b_i)=\sum f(\varepsilon(a_i))f(b_i).\]
Moreover, we have $f(\varepsilon(a_i))= e(a_i)$ since 
\[ f(\varepsilon(a_i))=0 \Longleftrightarrow \varepsilon(a_i)=0 \Longleftrightarrow a_i \in \Ker(\varepsilon) \Longleftrightarrow e(a_i)=0.\]
Therefore, $f(P) \in \sum f(\varepsilon(a_i))f(b_i)=\sum e(a_i)f(b_i)$. This shows that $f \in e *f$.\\
Next, we claim that if $g \in e*f$, then $g(P)=f(P)$ $\forall P \in k[X_{ij}]$ ($P$ does not contain a term involving $1/d$). Take such $P$ and let $\Delta P = \sum a_t \otimes_k b_t$ be a decomposition. Let $\delta_{ij}$ be the Kronecker delta. Then, we can write $a_t$ as $a_t=\alpha_t + \beta_t$, where $\alpha_t =\sum_l [b_l \prod_{i,j} (X_{ij} - \delta_{ij})^{m_{l,i,j}}]$ for some $b_l \in k$, $m_{l,i,j} \in \mathbb{Z}_{>0}$, and $\beta_t \in k$.
Then, since $\beta_t \in k$, it follows that
\[\Delta P = \sum (\alpha_t + \beta_t) \otimes_k b_t=\sum \alpha_t \otimes_k b_t + \sum \beta_t \otimes_k b_t = \sum \alpha_t \otimes_k b_t + 1 \otimes_k (\sum \beta_t b_t). \]
However, since the ideal $<X_{ij} - \delta_{ij}>$ is contained in $\Ker(e)$, we have $e(\alpha_t)=0$ $\forall t$. This implies that for this specific decomposition $\Delta P=\sum \alpha_t \otimes_k b_t + 1 \otimes_k (\sum \beta_t b_t)$, we have
\[\sum e(\alpha_t)f(b_t) + e(1)f(\sum \beta_t b_t)=f(\sum \beta_t b_t).\]
Therefore, we have $g(P)=f(P)=f(\sum \beta_t b_t)$ since $g,f \in e*f$. In general, for $q \in A=k[X_{ij}, 1/d]$, there exists $N \in \mathbb{N}$ such that $d^Nq \in k[X_{ij}]$. Then, from the previous claim, we have
\[f(d^N)f(q)=f(d^N q)=g(d^N q)=g(d^N)g(q).\]
However, since $d$ is invertible, we have $f(d^N)=f(d)^N=g(d^N)=g(d)^N=1$. It follows that $f(q)=g(q)$ $\forall q \in k[X_{ij},1/d]=A$. Thus $f=g$, and $\{f\}=e*f$. Similarly, one can show that $\{f\}=f*e$. This completes our proof.
\end{proof}

\noindent Next, we prove the first part of (\ref{statement3}): the existence of a canonical inverse.

\begin{lem}\label{inverse}
Let $S:A\longrightarrow A$ be the antipode map and $\varphi: k \rightarrow k/k^{\times}=\mathbf{K}$ is a canonical projection map. Then, for $f \in GL_n(\mathbf{K})$, we have $e= \varphi  \circ \varepsilon \in  (f*\tilde{f}) \cap (\tilde{f}*f)$, where $\tilde{f}=(f \circ S)$.
\end{lem}

\begin{proof}
Let $f \in \Hom(A,\mathbf{K})$ and $\tilde{f}=f \circ S$. Suppose that $a \in A$. Then, for a decomposition $\Delta a = \sum a_i \otimes_k b_i$, we have $\varepsilon(a) = \sum a_i S(b_i)$ since $\varepsilon$ is the counit and $S$ is the antipode map. This implies that
\begin{equation}\label{antipode}
f(\varepsilon(a))=f(\sum a_i S(b_i)) \in \sum f(a_i S(b_i))=\sum f(a_i)f(S(b_i))=\sum f(a_i) \tilde{f}(b_i).
\end{equation}
However, we know that $f(\varepsilon(a))=1$ if $\varepsilon(a)$ is non-zero and $f(\varepsilon(a))=0$ if $\varepsilon(a)$ is zero. Since $e=\varphi\circ \varepsilon$, it follows that $e(a)=\varphi(\varepsilon(a))=f(\varepsilon(a))$. Hence, the above \eqref{antipode} becomes 
\[
e(a) \in \sum f(a_i) \tilde{f}(b_i).
\] 
This shows that $e \in f*\tilde{f}$. Similarly, one can show that $e \in \tilde{f}*f$.
\end{proof}

\noindent Now we prove the last half of (\ref{statement3}): the inversion property.

\begin{lem}
Let $S:A\longrightarrow A$ be the antipode map and $f,g,h \in \Hom(A,\mathbf{K})$. Let $\tilde{f}=f \circ S$, $\tilde{g}= g \circ S$, $\tilde{h}= h \circ S$.
Then, $h \in f * g$ if and only if $\tilde{h} \in \tilde{g} * \tilde{f}$.
\end{lem}

\begin{proof}
Suppose that $\tilde{h} \in \tilde{g} * \tilde{f}$. Let $a \in A$ and $\Delta a = \sum a_i \otimes_k b_i$ be a decomposition of $a$. Let $t:A\otimes_k A \longrightarrow A\otimes_k A$ be the twist homomorphism; i.e., $t(a\otimes_k b)=b \otimes_k a$. Since $\Delta \circ S=t\circ (S\otimes_k S) \circ \Delta$, we have 
\begin{equation}
\Delta ( S(a)) = \sum S(b_i) \otimes_k S(a_i).
\end{equation} 
Since $S^2=id$, this implies that 
\begin{equation}
\tilde{h} (S(a)) \in \sum \tilde{g}(S(b_i)) \tilde{f}(S(a_i))=\sum \tilde{f}(S(a_i)) \tilde{g}(S(b_i)).
\end{equation}
 However, we have $\tilde{h}(S(a))=h \circ S (S(a)) = h(a)$. Similarly, $\tilde{g}(S(b_i))=g(b_i)$ and $\tilde{f}(S(a_i))=f(a_i)$. Thus, $h(a) \in \sum f(a_i)g(b_i)$. This shows that $h \in f*g$.\\ 
Conversely, suppose that $h \in f*g$. Then, for $a \in A$ and a decomposition $\Delta a = \sum a_i \otimes_k b_i$, we have $\tilde{h}(a) \in \tilde{g}(b_i) \tilde{f}(a_i)$. However, by the exact same argument as above and the fact that $S=S^{-1}$, one can conclude that $\tilde{h} \in \tilde{g}*\tilde{f}$.
\end{proof}

\noindent Finally, we prove (\ref{statement1}): the hyperoperation $*$ on $\Hom(A,\mathbf{K})$ is weakly-associative.

\begin{lem}
Let $A$ be a Hopf algebra over a field $k$, $\Delta$ be a coproduct of $A$, and $H:=(\Delta \otimes id) \circ \Delta = (id \otimes \Delta) \circ \Delta:A\longrightarrow A\otimes_k A\otimes_k A$. For $f,g,h \in \Hom(A,\mathbf{K})$, we let $J:=\Ker(f) \otimes_k A \otimes_k A + A\otimes_k \Ker(g) \otimes_k A +A\otimes_k A \otimes_k \Ker(h)$. Then, the set $P:=H^{-1}(J)$ is a proper prime ideal of $A$. Moreover, if $\varphi$ is an element of $\Hom(A,\mathbf{K})$ determined by $P$, then $\varphi \in f*(g*h)\cap (f*g)*h$.
\end{lem}

\begin{proof}
The proof is similar to Lemma \ref{nontrivialop}. For the first assertion, since $P$ is clearly an ideal by being an inverse image of an ideal, we only have to prove that $P$ is prime. Let $\alpha\beta \in P$. Then, since $H(\alpha\beta) \in J$, for any decomposition $H(\alpha\beta)=\sum \gamma_{(1)}\otimes_k\gamma_{(2)}\otimes_k \gamma_{(3)}$, we have \begin{equation}\label{Hprime}
\sum f(\gamma_{(1)})g(\gamma_{(2)})h(\gamma_{(3)})=0. 
\end{equation}
Suppose that $\alpha, \beta \not \in P$. Then, there exist decompositions $H(\alpha)=\sum a_i \otimes_k b_i \otimes_k c_i$ and $H(\beta)=\sum x_j \otimes_k y_j \otimes_k z_j$ such that 
\begin{equation}\label{impossible}
\sum f(a_i)g(b_i)h(c_i)=1 \textrm{ or } \{0,1\},\quad \sum f(x_j)g(y_j)h(z_j)=1 \textrm{ or }\{0,1\}. 
\end{equation}
With these two specific decompositions, we have
\[H(\alpha\beta)=H(\alpha)H(\beta)=(\sum_i a_i \otimes_k b_i \otimes_k c_i)(\sum_j x_j \otimes_k y_j \otimes_k z_j)=\sum_{i,j}a_ix_j \otimes_k b_iy_j \otimes_k c_iz_j.\]
Since $\alpha\beta \in P$, we should have
\begin{multline}\label{conclusion}
\sum_{i,j} f(a_ix_j)g(b_iy_j)h(c_iz_j)=\sum_{i,j}f(a_i)g(b_i)h(c_i)f(x_j)g(y_j)h(z_j)\\
=\sum_i[f(a_i)g(b_i)h(c_i)\sum_jf(x_j)g(y_j)h(z_j)]=0.
\end{multline}
However, \eqref{conclusion} contradicts to \eqref{impossible}. It follows that $\alpha \in P$ or $\beta \in P$. Furthermore, since $H(1)=1\otimes 1\otimes 1 \not \in J$, $P$ is proper. This proves the first assertion.\\
For the second assertion, it is enough to show that $\varphi \in f*(g*h)$ since the argument for $\varphi \in (f*g)*h$ will be symmetric. Let $\psi \in g*h$ such that $\Ker(\psi)=\Delta^{-1}(\Ker(g) \otimes_k A + A \otimes_k \Ker(h))$. This choice is possible by Lemma \ref{nontrivialop}. We claim that $\varphi \in f*\psi$. Indeed, we have to check two cases. The first case is when $a \in A$ has a decomposition $\sum a_i \otimes_k b_i$ such that $\sum f(a_i)\psi(b_i)=0$. Then, we have to show that $\varphi(a)=0$. But, since $\sum f(a_i)\psi(b_i)=0$, we know that $\sum a_i \otimes_k b_i \in \Ker(f) \otimes_k A + A \otimes_k \Ker(\psi)$. Since $\Ker(\psi)=\Delta^{-1}(\Ker(g)\otimes_k A+A\otimes_k \Ker(h))$, we have 
\[H(a)=(id \otimes_k \Delta)(\sum a_i \otimes_k b_i) \in \Ker(f) \otimes_k A \otimes_k A +A \otimes_k \Ker(g) \otimes_k A + A\otimes_k A \otimes_k \Ker(h).
\] 
Thus, $\varphi(a)=0$ since $\varphi$ is an element of $\Hom(A,\mathbf{K})$ which is determined by $H^{-1}(P)$. The second case is when $a \in A$ has a decomposition $\sum x_j \otimes_k y_j$ such that $\sum f(x_j)\psi(y_j)=1$. In this case, there exist $x_i$, $y_i$ such that $f(x_i)=\psi(y_i)=1$ and $f(x_j)\psi(y_j)=0$ $\forall j \neq i$. We may assume that $i=1$. Then, we have
\[\sum_{i \geq 2} x_i \otimes_k y_i \in \Ker(f) \otimes_k A + A \otimes_k \Ker(\psi). 
\]
This implies that $(id \otimes_k \Delta)(\sum_{i \geq 2} x_i \otimes_k y_i) \in J$. On the other hand, $(id \otimes_k \Delta)(x_1 \otimes_k y_1) \not\in J$ since $x_1 \not \in \Ker(f)$ and $y_1 \not \in \Ker(\psi)$. It follows that $H(a) \not \in J$, hence $\varphi(a)=1$ as we desired. The last case is when for any decomposition $\sum x_j \otimes_k y_j$ of $a$, we have that $\sum f(x_j)\psi(y_j)=\{0,1\}$. In this case, clearly we have $\varphi(a)= \sum f(x_j)\psi(y_j)$. This completes our proof.
\end{proof}

\noindent By combining the above lemmas, we obtain the following result.

\begin{mythm}\label{chevalleytheorem}
Any affine algebraic group scheme $X=\Spec A$ over a field $k$ has a canonical hyperstructure $*$ induced from the coproduct of $A$ which is weakly-associative and it is equipped with the identity element $e$. For each $f \in X$, there exists a canonical element $\tilde{f} \in X$ such that $e \in (f*\tilde{f})\cap (\tilde{f} * f)$. Furthermore, for $f,g,h \in X$, the following holds: $f \in g*h \Longleftrightarrow \tilde{f} \in \tilde{h}*\tilde{g}$.
\end{mythm}

Finally, we pose the following question.

\begin{que}
When $X=\mathbb{A}^1$ or $X=\mathbb{G}_m$, Connes and Consani's result (Theorem \ref{mainthemrem}) provides a nice description of the hypergroup structure in terms of the set of geometric points under the action of the absolute Galois group. Can we find a similar result with different affine algebraic group schemes?
\end{que}

\bibliography{Affine_group_hyper}\bibliographystyle{plain}

\end{document}